\documentclass[preprint,12pt]{elsarticle_arxiv}

\usepackage{amssymb}
\usepackage{amsmath}
\usepackage{amsthm}
\usepackage{enumerate}
\usepackage{tikz}
\newtheorem{definition}{Definition}
\newtheorem{theorem}{Theorem}
\newtheorem{lemma}{Lemma}
\newtheorem{remark}{Remark}
\newtheorem{corollary}{Corollary}
\newtheorem{proposition}{Proposition}
\newtheorem{assumption}{Assumption}

\renewcommand{\>}{\rangle}

\begin{document}

\begin{frontmatter}

\title{An explicit link between graphical models and Gaussian Markov random fields on metric graphs}

\author[kaust]{David Bolin} 
\ead{david.bolin@kaust.edu.sa}
\corref{cor1}
\cortext[cor1]{Corresponding author.}

\author[kaust]{Alexandre B. Simas}
\ead{alexandre.simas@kaust.edu.sa}

\author[lund]{Jonas Wallin}
\ead{jonas.wallin@stat.lu.se}

\affiliation[kaust]{organization={Statistics Program, Computer, Electrical and Mathematical Sciences and Engineering Division}, 
addressline={King Abdullah University of Science and Technology (KAUST)}, city={Thuwal}, country={Saudi Arabia}}

\affiliation[lund]{organization={Department of Statistics}, addressline={Lund University}, city={Lund}, country={Sweden}}

\begin{abstract}
	We derive an explicit link between Gaussian Markov random fields on metric graphs and graphical models, and in particular show that a Markov random field restricted to the vertices of the graph is, under mild regularity conditions, a Gaussian graphical model with a distribution which is faithful to its pairwise independence graph, which coincides with the neighbor structure of the metric graph. 
	This is used to show that there are no Gaussian random fields on general metric graphs which are both Markov and isotropic in some suitably regular metric on the graph, such as the geodesic or resistance metrics. 
	\end{abstract}

\end{frontmatter}



\section{Introduction}
There is a growing interest in modeling data on networks and graphs in statistics, machine learning, and signal processing. 
In statistics, typical applications are modeling of traffic accidents and street crimes on road networks \cite{CHAUDHURI2023, MORADI2024} and  modeling of environmental variables such as temperature and pollutants on river networks \cite{Gardner2003, Isaak2014, mcgill2021spatiotemporal}. In machine learning and graph signal processing \cite{ortega2018graph}, common areas of application are social networks \cite{Perozzi2014}, energy networks \cite{He2018}, sensor networks \cite{Wagner2005}, transportation networks \cite{Valdivia2015}, and brain networks \cite{Bullmore2012}. 
In several of these applications, the quantity of interest is defined on both the vertices and edges of the graph, where the edges typically are curves such as river or road segments and the vertices are the intersections between these. The natural distance between measurement locations is in this case not the Euclidean distance but the geodesic (shortest path) distance on the network. A graph like this, equipped with the geodesic distance, is referred to as a metric graph. 

Specifically, a compact metric graph $\Gamma$ consists of finitely many vertices $\mathcal{V} = \{v_i\}$ and edges $\mathcal{E} = \{e_j\}$, where each edge $e$ is a rectifiable curve with finite length $l_e < \infty$ parameterized by arc length, connecting a pair of vertices. A location $s \in \Gamma$ is a position on an edge which can be represented as $(t, e)$, where $t \in [0, l_e]$. The graph is equipped with a metric $d(\cdot, \cdot)$, such as the geodesic metric which is well-defined as long as $\Gamma$ is connected.
An important special case of compact metric graphs is linear networks, where each edge is a straight line. Another special case is metric trees, which are graphs without loops. A final, commonly studied, special case involves graphs with Euclidean edges \cite{anderes2020isotropic}, which are such that there is at most one edge between any pair of vertices, no edges connecting a vertex to itself, and distances between connected vertices equal to the lengths of the corresponding edges.

Gaussian random fields are intrinsic to dealing with spatially correlated data, and due to the growing number of statistical applications involving data on networks there has been significant recent research focused on defining random fields on compact metric graphs and various special cases such as graphs with Euclidean edges and metric trees \cite{CressieRiver, Hoef2006, Hoef2010, anderes2020isotropic, BSW2022, moller2022lgcp, bolinetal_fem_graph}. 
An alternative which is also often 
used in machine learning and signal processing 
is to simply
 ignore the continuous structure of the space and only consider a process
  $u_{\mathcal{V}}$ defined at the vertices of the graph. In this case, it is natural to consider Gaussian graphical models for $u_{\mathcal{V}}$, 
  where the edges of the graph are used to define the conditional dependence structure of $u_{\mathcal{V}}$ \citep{borovitskiy2021matern,Perozzi2014}. 

A natural question, which is the main focus of this work, is if there is any connection between these two approaches. 
Specifically, our main goal is to establish an explicit link between continuously indexed Gaussian Markov random fields (GMRFs) on metric graphs and discrete graphical models \citep{rue2005gaussian}. 
For practical applications, this is important because it allows us to determine whether a graphical model could be defined in such a way that it corresponds to the restriction of a continuously indexed Gaussian random field on the metric graph to the vertices. Further, if a continuously indexed Gaussian random field corresponds to a Gaussian graphical model when evaluated at some finite number of locations, this means that we can use the vast literature on inference procedures for Gaussian graphical models when fitting the model to data. 

We answer this question by deriving a link to the pairwise independence graph (see Section~\ref{sec:review_markov} or \citep{lauritzen1996graphical,rue2005gaussian}) of the process, induced by its distribution on a set of finite points on the metric graph.
More precisely, consider a metric graph $\Gamma$ with an arbitrary metric $d(\cdot,\cdot)$, a GMRF $u$ of order 1 on $\Gamma$ with a continuous strictly positive definite covariance function and a finite collection of points $\mathcal{P}\subset \Gamma$ that include the vertices of $\Gamma$, and that has at least one point in the interior of each multiple edge of $\Gamma$. Let $P_{\mathcal{P}}$ be the distribution of $\{u(p):p\in \mathcal{P}\}$. We will later show that $P_{\mathcal{P}}$ is faithful (see Section~\ref{sec:review_markov} or \citep{lauritzen1996graphical,fallat2017total}) to its pairwise independence graph $(\mathcal{P}, \mathcal{E}(P_{\mathcal{P}}))$.

Given $\mathcal{P}\subset \Gamma$, let $\Gamma_{\mathcal{P}}$ be the combinatorial graph $(\mathcal{P}, \mathcal{E}_{\mathcal{P}})$, where $\mathcal{E}_{\mathcal{P}}$ are the edges obtained from $\Gamma$ after turning the points in $\mathcal{P}$ into vertices (that is, it splits the edges in which the points $p$ which are in the interior of edges). The faithfulness of $P_{\mathcal{P}}$ in $(\mathcal{P}, \mathcal{E}(P_{\mathcal{P}}))$ implies, in particular, that given $v_i,v_j\in\mathcal{P}$, the points $v_i$ and $v_j$ are neighbors in $\Gamma_{\mathcal{P}}$ if, and only if, they are neighbors in $(\mathcal{P}, \mathcal{E}(P_{\mathcal{P}}))$. This creates an explicit link between GMRFs of order 1 and graphical models, that is, it shows that if we consider a finite collection of points in the metric graph that contains the vertices, and that has at least one point in the interior of each multiple edge of $\Gamma$, the neighboring metric graph structure coincides with neighboring structure of its pairwise independence graph and, even more, its distribution is faithful to its pairwise independence graph. This is our main result and in particular means that theoretical results for graphical models can be applied to GMRFs on metric graphs. 

As an important application of this result, we show that  
that under a weak homogeneity condition, there are no GMRFs of order 1 that are simultaneously Markov and isotropic on general metric graphs. Here, isotropic means that the Gaussian process has a covariance function of the form $r(\cdot, \cdot) = \rho(d(\cdot,\cdot))$, where $\rho$ is an isotropic covariance function and $d$ is a metric on the graph. This construction was studied by \citet{anderes2020isotropic} where they showed that several choices of $\rho$ are possible if the graph is restricted to have Euclidean edges and the metric is chose as the resistance metric \citep[see][]{anderes2020isotropic}. For this construction, our result shows that any isotropic Gaussian field (using any metric such as the the geodesic metric or the resistance metric) constructed through the approach of \citep{anderes2020isotropic} can only be Markov if the graph is either a metric tree or a cycle. For any application, one must therefore choose between the two properties when determining which model to use.  

The outline of the paper is as follows. In Section \ref{sec:review_markov}, basic definitions and results regarding Gaussian Markov random fields are recalled. In Section \ref{sec:isotropic}, the main result containing the explicit link is stated and a few consequences of it are discussed. The proof of the main result is given in Section~\ref{sec:proof} together with a few additional results. In Section \ref{sec:illustration}, an example and the result about isotropic fields are presented. The paper ends in Section \ref{sec:discussion} with a discussion on extensions and open problems. 

\section{Gaussian graphical models and GMRFs on metric graphs}\label{sec:review_markov}
In this section, we provide a brief review of Gaussian graphical models and GMRFs on compact metric graphs. 

\subsection{Gaussian graphical models}
In this subsection, we introduce the notation we need about Gaussian graphical models. 
For a comprehensive treatment of Gaussian graphical models, we refer to \cite{lauritzen1996graphical} or \cite{rue2005gaussian}. We begin with the definition of Gaussian graphical model with respect to an undirected graph $(\mathcal{V}, \mathcal{E})$, where the set of vertices $\mathcal{V}$ is finite. 

\begin{definition}\label{def:ggm}
	Let $\mathcal{V} = \{v_i\}$ be a finite set of vertices and $U$ be a random vector indexed by $\mathcal{V}$ and following a multivariate Gaussian distribution with mean vector $\mu$ and strictly positive-definite covariance matrix $\Sigma$. Further, let $Q = \Sigma^{-1} = \{Q_{v_iv_j}\}$ be its precision matrix. A graphical model for $U$ with respect to the graph $(\mathcal{V},\mathcal{E})$, where $\mathcal{E}$ is a finite set of edges with respect to $\mathcal{V}$, is given by assuming that 
	$Q_{v_iv_j} \neq 0$ if, and only if, there exists an edge in $\mathcal{E}$ connecting $v_i$ to $v_j$.
\end{definition}

A notable feature of Gaussian graphical models is that the pairwise interactions between their coordinates are sufficient to fully characterize the dependency structure of the model. Consequently, the pairwise independence graph, which we will define shortly, plays a crucial role in the theory of Gaussian graphical models.

\begin{definition}\label{def:indep_graph}
	Given a finite set of vertices $\mathcal{V}=\{v_i\}$, a probability distribution $P$ on $\mathcal{V}$, and a random vector $U$ following the distribution $P$, the pairwise independence graph of $P$ is defined as $(\mathcal{V}, \mathcal{E}(P))$, where $v_i$ and $v_j$ are not neighbors in $\mathcal{E}(P)$, for $v_j,v_j\in\mathcal{V}$, if, and only if, $U(v_i)$ and $U(v_j)$ are conditionally independent given $\{U(v) : v\in \mathcal{V}\setminus\{v_i,v_j\}\}$. 
\end{definition}

It is well-known that a random vector \( U \) indexed by a finite set of vertices \( \mathcal{V} \) and following a multivariate Gaussian distribution is always a graphical model with respect to its pairwise independence graph. Specifically, let \( \mathcal{V} \) be a finite set of vertices, and let \( U \) be a random vector indexed by \( \mathcal{V} \) following a multivariate Gaussian distribution \( P_G \). 
Given two vertices \( v_i \) and \( v_j \) in \( \mathcal{V} \), and a collection of vertices \( \widetilde{\mathcal{V}} \), one might intuitively expect that \( U(v_i) \) and \( U(v_j) \) are conditionally dependent given \( \{U(v) : v \in \widetilde{\mathcal{V}}\} \) if there is a path connecting \( v_i \) and \( v_j \) that does not intersect \( \widetilde{\mathcal{V}} \) in the pairwise independence graph \( (\mathcal{V}, \mathcal{E}_{P_G}) \). However, this intuition does not always hold true, which motivates the definition of faithfulness, which we provide below.

\begin{definition}
	Let $\mathcal{V}$ be a finite collection of vertices, let $P$ be a probability distribution on $\mathcal{V}$ and let $U$ be a random vector with distribution $P$. We say that $P$ is faithful to its pairwise independence graph $(\mathcal{V}, \mathcal{E}(P))$ if, given $v_i,v_j\in\mathcal{V}$ and $S\subset\mathcal{V}$,
	$U(v_i)$ is conditionally independent of $U(v_j)$ given $\{U(s):s\in S\}$  if, and only if, $v_i$ and $v_j$ are separated by $S$ in $(\mathcal{V}, \mathcal{E}(P))$, that is, if, and only if, every path from $v_i$ to $v_j$ in  $(\mathcal{V}, \mathcal{E}(P))$ intersects $S$.
\end{definition}

Another important property related to faithfulness is MTP$_2$, defined as follows.

\begin{definition}\label{def:mtp2}
	Let $U$ be a Gaussian random vector indexed by a finite set $S\neq\emptyset$ with an invertible covariance matrix $\Sigma$ and let $Q = \Sigma^{-1} = [Q_{ts}]_{t,s\in S}$ be its precision matrix. The distribution of $U$ is  MTP$_2$ if for all $s\in S$, $Q_{ss} > 0$, and if for every pair $s,t\in S$, $Q_{st} \leq 0.$
\end{definition}

We refer the reader to \cite{lauritzen1996graphical,fallat2017total} for further details. It should be noted that the definition above sometimes is stated as a theorem characterizing Gaussian MTP$_2$ distributions in terms of a more general definition of the MTP$_2$ property. We will, however, not need that generality here.
The following result is a direct consequence of \cite[Proposition G.1]{slawski2015estimation} (see, also, \cite[Theorem~6.1]{fallat2017total} and \cite[Proposition 3.1]{lauritzen1996graphical}).

\begin{theorem}\label{thm:faithfulmtp2}
Let $U$ be a Gaussian random vector indexed by a finite set $S\neq\emptyset$ with an invertible covariance matrix $\Sigma$ and let $P$ be the distribution of $U$. If $P$ is MTP$_2$, then $P$ is faithful to its pairwise independence graph $(S, \mathcal{E}(S))$.
\end{theorem}

\subsection{GMRFs on metric graphs}
In this subsection, we use the definitions and notations from \cite{BSW_Markov}, which we refer to for a comprehensive treatment on the subject.

Let us begin by establishing some notation and basic definitions that will be used throughout this paper. Since \(\Gamma\) is a metric space when equipped with the geodesic metric \(d(\cdot, \cdot)\), we will work with its associated metric topology.
For any subset \( S \subset \Gamma \), we let \( \partial S \) denote the topological boundary of \( S \), and \( \overline{S} \) denote the closure of \( S \). Additionally, given any $\varepsilon > 0$, the \(\varepsilon\)-neighborhood of \( S \) is defined as 
\[ S_\varepsilon = \{s \in \Gamma : \exists z \in S \text{ such that } d(s, z) < \varepsilon\},\] 
where \( d(s, z) \) is the geodesic distance on \(\Gamma\).

Let \((\Omega, \mathcal{F}, \mathbb{P})\) be a complete probability space, where \(\Omega\) is the sample space, \(\mathcal{F}\) is the \(\sigma\)-algebra of events, and \(\mathbb{P}\) is the probability measure. The expectation of a real-valued random variable \( Z \) is denoted by 
\[ \mathbb{E}(Z) = \int_\Omega Z(\omega) d\mathbb{P}(\omega). \]

For sub-\(\sigma\)-algebras \(\mathcal{A}\), \(\mathcal{B}\), and \(\mathcal{C}\) of \(\mathcal{F}\), we say that \(\mathcal{A}\) and \(\mathcal{B}\) are conditionally independent given \(\mathcal{C}\) if, for every \(A \in \mathcal{A} \) and \( B \in \mathcal{B} \), we have that 
\({
\mathbb{P}(A \cap B | \mathcal{C}) = \mathbb{P}(A | \mathcal{C}) \mathbb{P}(B | \mathcal{C}),
}\)
where \(\mathbb{P}(F | \mathcal{C}) = \mathbb{E}(1_F | \mathcal{C})\) is the conditional probability of a set \( F \in \mathcal{F} \) given \(\mathcal{C}\). In this case, we say that \(\mathcal{C}\) splits \(\mathcal{A}\) and \(\mathcal{B}\).

Now, for a stochastic process \( u \) defined on the compact metric graph \(\Gamma\), we introduce the \(\sigma\)-algebras 
\[
\mathcal{F}^u(S) = \sigma(u(s) : s \in S),\quad\text{and}\quad
\mathcal{F}^u_+(S) = \bigcap_{\varepsilon > 0} \mathcal{F}^u(S_\varepsilon),
\] 
where $\mathcal{F}^u(S)$ represents the smallest \(\sigma\)-algebra containing all information about the process \( u \) on the set \( S \) and $\mathcal{F}^u_+(S)$ captures the behavior of \( u \) around an ``infinitesimal'' neighborhood of \( S \).
With these notations and definitions in place, we can now define Markov random fields (MRFs) on compact metric graphs, as introduced in \cite{BSW_Markov}, which builds upon the framework developed in \cite{kunsch}.

\begin{definition}\label{def:MarkovPropertyField}
	A random field $u$ on a compact metric graph $\Gamma$ is an MRF (of any possible order) if for every open set $S$, $\mathcal{F}^u_+(\partial S)$ splits $\mathcal{F}^u_+(\overline{S})$ and $\mathcal{F}^u_+(\Gamma \setminus S)$.
\end{definition} 

To apply this definition effectively, it is crucial to characterize the splitting $\sigma$-algebra $\mathcal{F}^u_+(\partial S)$, as it contains the information required to ensure the independence of the field on $\bar{S}$ from the field on $\Gamma \setminus S$. This characterization becomes especially important for differentiable random fields on $\Gamma$, which we will now define. 

To this end, we first need to introduce Gaussian spaces. Let $L_2(\Omega)$ be the Hilbert space of square-integrable real-valued random variables. Given a GRF $u$ on $\Gamma$ and a Borel set $S \subset \Gamma$, the Gaussian space is defined as
\[
H(S) = \overline{\text{span}(u(s) : s \in S)},
\]
where the closure is taken with respect to the $L_2(\Omega)$-norm. Since each element in ${\text{span}(u(s) : s \in S)}$ is Gaussian, the space $H(S)$ consists of $L_2(\Omega)$-limits of these Gaussian variables. Therefore, $H(S)$ contains only Gaussian random variables \cite[e.g., see][Theorem 1.4.2, p.39]{ashtopics}.

Let \( u \) be a Gaussian random field with associated Gaussian space \( H(\Gamma) \), \( e \) be an edge, and \( v: e \to H(\Gamma) \) be a function. We say \( v \) is weakly differentiable at \( s = (t,e) \in \Gamma \) in the \( L_2(\Omega) \) sense if there exists an element \( v'(s) \in H(\Gamma) \), that will be referred to as a weak derivative, such that for every \( w \in H(\Gamma) \) and any sequence \( s_n \to s \) with \( s_n \neq s \), we have
\[
\mathbb{E}[w(v(s_n) - v(s))/(s_n - s)] \to \mathbb{E}[wv'(s)].
\]
Higher-order weak derivatives are defined inductively: for \( k \geq 2 \), \( v \) has a \( k \)-th order weak derivative at \( s = (t,e) \in \Gamma \) if \( v_e^{(k-1)}(\tilde{t}) \) exists for all \( \tilde{t} \in [0, \ell_e] \) and \( v_e^{(k-1)} \) is weakly differentiable at \( t \). Finally, a function \( v: e \to H(\Gamma) \) is weakly continuous in the \( L_2(\Omega) \) sense if for every \( w \in H(\Gamma) \), the function \( s \mapsto \mathbb{E}[w v(s)] \) is continuous.

\begin{definition}\label{def:RandomFieldOrderP}
	Let $u$ be a GRF on $\Gamma$ that has weak derivatives, in the 
	$L_2(\Omega)$ sense, of orders $1,\ldots, p$ for $p\in\mathbb{N}$.
	Also assume that $u_e^{(j)}$, for $j=0,\ldots,p$ and $e\in\mathcal{E}$, is weakly continuous in the 
	$L_2(\Omega)$ sense.
	Then the weak derivatives are well-defined for each $s\in\Gamma$ and we say that $u$ is a differentiable Gaussian 
	random field of order $p$. For convenience, a GRF that is continuous in $L_2(\Omega)$ is said to be a differentiable GRF of order $0$.
\end{definition} 

We can now introduce higher-order Markov properties as follows \citep{BSW_Markov}. 

\begin{definition}\label{def:MarkovPropertyFieldOrderP}
	A random field $u$ on a compact metric graph $\Gamma$ is Markov of order $p$ if it has $p-1$ weak derivatives in the mean-squared sense (see Definition~\ref{def:RandomFieldOrderP}) and
	$\sigma(u(s), u'(s),\ldots, u^{(p-1)}(s): s\in\partial S)$ splits $\mathcal{F}_+^u(\overline{S})$ and $\mathcal{F}_+^u(\Gamma \setminus S)$ for every open set $S$.
\end{definition}

Note that every MRF of order $p$ is also an MRF according to Definition~\ref{def:MarkovPropertyField}. 

Finally, let us recall the definition of the Cameron--Martin spaces, which are key results in order to study the Markov property of Gaussian random fields. For a GRF $u$ with a continuous covariance function ${\rho:\Gamma\times\Gamma\to\mathbb{R}}$, we define the Cameron--Martin space $\mathcal{H}(S)$, for $S\subset\Gamma$, as
$$
\mathcal{H}(S) = \{h(s) = \mathbb{E}(u(s)v): s\in \Gamma\hbox{ and } v\in H(S)\},
$$
with inner product $\langle h_1, h_2\rangle_{\mathcal{H}} = \mathbb{E}(v_1 v_2),$
where $h_j(s) = \mathbb{E}(u(s)v_j)$, $s\in\Gamma$, and $v_j \in H(S)$, where $j=1,2$.  
Furthermore, it is well-known that $\mathcal{H}(S)$ is isometrically isomorphic to $H(S)$, e.g., \cite[Theorem 8.15]{janson_gaussian}. For the following definition, recall that the support of a function $h: \Gamma \to \mathbb{R}$ is given by $\text{supp}\, h = \overline{\{s \in \Gamma : h(s) \neq 0\}}$.
\begin{definition}\label{def:localCMspaces}
	Let $\Gamma$ be a compact metric graph and let $H$ be an inner product space of functions defined on $\Gamma$, endowed with the inner product
	$\langle \cdot, \cdot\rangle_{H}$. We say that $H$ is local if:
	\begin{enumerate}[i]
		\item If $h_1,h_2\in H$ and $\textrm{supp}\, h_1 \cap \textrm{supp}\, h_2 = \emptyset$, then $\langle h_1,h_2\rangle_{H} = 0;$\label{def:localCMspaces1}
		\item If $h \in H$ is such that $h = h_1+h_2$, where $\textrm{supp}\, h_1\cap \textrm{supp}\,h_2 = \emptyset$, then $h_1, h_2 \in H$.\label{def:localCMspaces2}
	\end{enumerate}
\end{definition}

With this definition, we have the following important characterization.

\begin{theorem}[Theorem 1, \cite{BSW2022}]\label{thm:MarkovLocalCM}
	Let $u$ be a GRF defined on a compact metric graph $\Gamma$, with Cameron--Martin space $\mathcal{H}(\Gamma)$. Then, $u$ is a GMRF if, and only if, $\mathcal{H}(\Gamma)$ is local.
\end{theorem}

\section{Explicit link with graphical models}\label{sec:isotropic}

Let $\Gamma$ be a compact and connected metric graph with an arbitrary metric $d(\cdot,\cdot)$, and let $\mathcal{P} \subset\Gamma$ be a finite collection of points such that $\mathcal{P}\supset \mathcal{V}$. Recall that we define $\Gamma_{\mathcal{P}}$ as the combinatorial graph $(\mathcal{P}, \mathcal{E}_{\mathcal{P}})$, where $\mathcal{E}_{\mathcal{P}}$ is the set of edges induced by $\mathcal{P}$ on $\mathcal{E}$, after all points of $\mathcal{P}$ that were not vertices were turned into vertices. More precisely, $\mathcal{E}_{\mathcal{P}}$ contains an edge $e\in\mathcal{E}$ if $(e\cap \mathcal{P})\setminus \mathcal{V} = \emptyset$, and contains subintervals from edges of $\mathcal{E}$ otherwise, where the subintervals are the splits induced by the points of $\mathcal{P}$ (treating these points as vertices of degree 2). 
We introduce the notation 
$$
v_i \stackrel{\mathcal{E}_{\mathcal{P}}}{\sim} v_j
$$ 
for the case when $v_i,v_j\in\mathcal{P}$ are neighbors in $\Gamma_{\mathcal{P}}$, i.e., if $\mathcal{E}_{\mathcal{P}}$ contains an edge from $v_i$ to $v_j$, and we write $X\perp Y | Z$ if the random variables $X$ and $Y$ are conditionally independent given the (set of) random variables $Z$. The following definitions are also important:

\begin{definition}\label{def:admissible_points}
	We say that a finite collection of points $\mathcal{P}$ as above is admissible for $\Gamma$ if $\Gamma_{\mathcal{P}}$ does not contain multiple edges, i.e., there is at most one edge between any pair of vertices in $\Gamma_{\mathcal{P}}$.
\end{definition}

\begin{definition}\label{def:paths}
	Let $\Gamma$ be a metric graph and $s_1,s_2\in\Gamma$. We say that $p$ is a path on $\Gamma$ connecting $s_1$ and $s_2$ if there exist some natural number $n\in\mathbb{N}$, and edges ${e_1,\ldots,e_n \in\mathcal{E}}$, with the identification $e_j = [0,l_{e_j}]$, $j=1,\ldots,n$, such that 
	\begin{enumerate}[i]
	\item ${p = f_1 \cup e_2\cup\cdots\cup e_{n-1} \cup f_n}$, where $f_1 = [0,s_1]$ or $f_1=[s_1,l_{e_1}]$, and ${f_n = [0,s_2]}$ or ${f_n = [s_2,l_{e_n}]}$, 
	\item for $j=2,\ldots,n$, we have either $(0,e_j) = (1,e_{j-1})$, $(0,e_j) = (0,e_{j-1})$, $(1,e_j) = (0,e_{j-1})$ or $(1,e_j) = (1,e_{j-1})$. 
	\end{enumerate}
	Further, we say that a path $p$ is simple if each vertex in $p$ belongs to at most two edges, and $p$ does not contain loops. 
\end{definition}

We will also need the following regularity condition on the metric of the metric graph:

\begin{definition}\label{def:intermediatevalue}
Let $\Gamma$ be a compact metric graph with a metric $d(\cdot,\cdot)$. 
\begin{enumerate}[i]
	\item If for every $e\in\mathcal{E}$, and every $t, t_1, t_2\in e$ and $\tau>0$ such that $d(t,t_1) > \tau$ and $d(t,t_2)<\tau$, a $t^*\in e$ exists such that $d(t,t^*) = \tau$, we say that $d(\cdot,\cdot)$ satisfies the intermediate value property. 
	\item Let $d(\cdot,\cdot)$ be a metric that satisfies the intermediate value property on $\Gamma$. Given $t,s\in\Gamma$, any simple path $P_{t,s}$ (i.e., the path does not have self intersections) between $t$ and $s$, and any $0 < \tau < d(t,s)$, consider the ordered sequence of points $p_k\in P_{t,s}$, $k\in\mathbb{N}$, defined as 
	\begin{itemize}
	\item $p_0 = s$, and $p_1 \in P_{t,s}$ is such that $d(p_1,s) = \tau$; 
	\item $p_2 \in p_{t, p_1} \subset P_{t,s}$ is such that $d(p_1,p_2) = \tau$; and 
	\item for any $k>1$, $p_k \in P_{t, p_{k-1}}\subset P_{t, p_{k-2}}$ is such that $d(p_k, p_{k-1}) = \tau$. 
	\end{itemize}
	We say that $d(\cdot,\cdot)$ is regular if for any $t,s\in \Gamma$ and any $0 < \tau < d(t,s)$, this sequence is finite, and $d(t, p_{N_{t,s}}) < \tau$, where $N_{t,s}$ is the last term of this sequence.
	\item We say that $d(\cdot,\cdot)$ is connected on $\Gamma$ if all simple paths on $\Gamma$ are connected sets on $(\Gamma,d)$.
\end{enumerate}
\end{definition}

\begin{remark}
	The geodesic distance and resistance distance satisfy the intermediate value property, are regular and are connected in the sense of Definition \ref{def:intermediatevalue}.
\end{remark}

We are now in a position to state our main theorem. 

\begin{theorem}\label{thm:explicit_link}
	Let $\Gamma$ be a compact and connected metric graph endowed with a metric $d(\cdot,\cdot)$ that satisfies the intermediate value property, is regular and connected in the sense of Definition \ref{def:intermediatevalue}. Let, also, $u(\cdot)$ be a GMRF of order 1 with a continuous strictly positive definite covariance function $r$. Let $\mathcal{P}$ be an admissible set of points for $\Gamma$. Then, the distribution of $\{u(s): s\in \mathcal{P}\}$ is MTP$_2$.
\end{theorem}

The proof of this result is provided in Section~\ref{sec:proof}. 
As an immediate corollary, GMRFs of order 1 on compact metric graphs induce graphical models that are faithful to their pairwise independence graph. This tells us that the Markov structure on metric graphs translates very well to graphical models, providing a sound dependence structure, so that, what we obtain matches the intuition, that communication on the graph means (conditional) dependence. Specifically, we have the following result which is an immediate consequence of Theorems \ref{thm:faithfulmtp2} and \ref{thm:explicit_link}.

\begin{corollary}\label{cor:faithfulness}
	Let $\Gamma$ be a compact metric graph endowed with a metric $d(\cdot,\cdot)$ that satisfies the intermediate value property, is regular and connected in the sense of Definition~\ref{def:intermediatevalue}. Let $u(\cdot)$ be a GMRF of order 1 with a covariance function that is continuous and strictly positive definite. Let $\mathcal{P}$ be an admissible set of points for $\Gamma$. Let $P_u$ be the distribution of $u$. For $t,s\in\mathcal{P}$ we have that 
	$$
	t\stackrel{\mathcal{E}_{\mathcal{P}}}{\sim}s \Longleftrightarrow t\stackrel{\mathcal{E}(\mathcal{P}_u)}{\sim}s ,
	$$
	where $(\mathcal{P},\mathcal{E}(P_u))$ is the pairwise independence graph of $P_u$. Furthermore, we have that $u(\mathcal{P}) := \{u(s): s\in \mathcal{P}\}$ is faithful to its pairwise independence graph, and in particular, we given $t,s\in\mathcal{P}$ and $S\subset \mathcal{P}$,
	$$
	u(t)\perp u(s) | \{u(w): w\in S\} \Longleftrightarrow \hbox{$S$ separates $t$ and $s$ in $\mathcal{E}_{\mathcal{P}}$}.$$
\end{corollary}

\begin{remark}
	Observe that Corollary \ref{cor:faithfulness} gives us that in order to check conditional independence, it is enough to only look at the points on the metric graph. There is no need to construct the pairwise independence graph as their connection structure coincides with the connection structure in $\mathcal{E}_{\mathcal{P}}$, which is the connection structure induced by the metric graph.
\end{remark}

\section{Additional results and the proof of Theorem~\ref{thm:explicit_link}}\label{sec:proof}
The proof of Theorem~\ref{thm:explicit_link} requires the following three lemmata, which are interesting statements on their own. In the following result, observe that $\stackrel{d}{=}$ means equality of the finite dimensional distributions. The following notation and auxiliary result will be needed for the first lemma.

Let $\mathcal{E}_s$ denote the set of edges incident to a point \( s \in \Gamma \). If \( s \) is an interior point of an edge \( e \), then \( \mathcal{E}_s \) contains only the edge \( e \). Additionally, for a set \( S \subset \Gamma \), let $\mathcal{E}_S$ be the set of edges that intersect the interior of \( S \). We also introduce
\begin{align*}
	H_\alpha(\partial S) &= \textrm{span}\{u_e(s), u_e'(s),\ldots, u_e^{(\alpha-1)}(s):
	s\in \partial S, e\in \mathcal{E}_s \},\\
	\mathcal{F}^u_{\alpha}(\partial S) &= \sigma(u_e(s), u_e'(s),\ldots, u_e^{(\alpha-1)}(s):
	s\in \partial S, e\in \mathcal{E}_s).
\end{align*}   

We have the following Proposition, proved in \cite{BSW_Markov}:

\begin{proposition}\label{prp:charmarkovspaces}
	Let $\Gamma$ be a metric graph and let $u$ be a GRF on $\Gamma$. Then, the fact that $u$ is an MRF is equivalent to any of the following statements: For any open set $S$, 
	\begin{enumerate}[i]
		\item the orthogonal projection of $H_+(\overline{S})$ on $H_+(\Gamma\setminus S)$ is $H_+(\partial S)$;\label{prp:charmarkovspaces1}
		\item $H(\Gamma) = H_+(\overline{S})\oplus(H_+(\Gamma\setminus S)\ominus H_+(\partial S)) = H_+(\Gamma\setminus S)\oplus (H_+(\overline{S})\ominus H_+(\partial S))$. \label{prp:charmarkovspaces2}
	\end{enumerate}
\end{proposition}

We are now in a position to prove the first lemma.

        \begin{lemma}\label{lem:condpoints_markov}
            Let $\Gamma$ be a compact metric graph, $\mathcal{P}\subset\Gamma$ be a finite collection of points, and $u$ be a GMRF of order 1 on $\Gamma$. Define $u_{\mathcal{P}}(\cdot)$ as the field obtained after conditioning $u(\cdot)$ on $\{u(p)=0: p\in\mathcal{P}\}$, that is, 
			\begin{equation}\label{eq:dist}
			u_{\mathcal{P}}(\cdot) \stackrel{d}{=} u(\cdot) | u(p) = 0: p\in\mathcal{P}. 
			\end{equation}
			Then, $u_{\mathcal{P}}(\cdot)$ is a GMRF of order 1. Furthermore, for every $p\in\mathcal{P}$,  $u_{\mathcal{P}}(p) = 0$ almost surely.
        \end{lemma}

		\begin{proof}
			Let $H(\Gamma) = \overline{\hbox{span}\{u(s):s\in\Gamma\}}$ be the closure, in $L_2(\Omega)$, of the linear span of $u$ in $\Gamma$. Furthermore, let $H_{\mathcal{P}}(\Gamma) = \hbox{span}\{u(p): p\in\mathcal{P}\}$ and observe that since $H_{\mathcal{P}}$ is finite-dimensional, it is a closed subspace of $H(\Gamma)$. By the orthogonal projection theorem we have that
		$H(\Gamma) = H_{\mathcal{P}}(\Gamma)^\perp \oplus H_{\mathcal{P}}(\Gamma).$
		We will now write the above identity in terms of Cameron-Martin spaces. More precisely, let $\mathcal{H}(\Gamma)$ be the Cameron-Martin space of $u$ and let ${\Phi:H(\Gamma)\to\mathcal{H}(\Gamma)}$ be the isometric isomorphism map given by $\Phi(Y)(s) := \mathbb{E}(Yu(s))$, $Y\in H(\Gamma)$. Then, we have that
		$\mathcal{H}(\Gamma) = \mathcal{H}_{\mathcal{P}}(\Gamma)^\perp\oplus \mathcal{H}_{\mathcal{P}}(\Gamma),$
		where $\mathcal{H}_{\mathcal{P}}(\Gamma) = \Phi(H_{\mathcal{P}}(\Gamma))$. 

		Let $r(\cdot,\cdot)$ be the covariance function of $u$, and let $\Pi_{\mathcal{P}}^\perp$ be the orthogonal projection operator onto $\mathcal{H}_{\mathcal{P}}(\Gamma)^\perp$. We have that for any $f\in \mathcal{H}_{\mathcal{P}}(\Gamma)^\perp$:
		$$\<f, \Pi_{\mathcal{P}}^\perp(r(\cdot, s))\>_{\mathcal{H}(\Gamma)} = \<f, r(\cdot, s)\>_{\mathcal{H}(\Gamma)} = f(s).$$
		Thus, $r_{\mathcal{P}}(\cdot,s) := \Pi_{\mathcal{P}}^\perp(r(\cdot, s)), s\in\Gamma$, is a reproducing kernel for $\mathcal{H}_{\mathcal{P}}(\Gamma)^\perp$, which shows that $\mathcal{H}_{\mathcal{P}}(\Gamma)^\perp$ is a reproducing kernel Hilbert space. Similarly, $\mathcal{H}_{\mathcal{P}}(\Gamma)$ is the finite-dimensional Hilbert space:
		$\mathcal{H}_{\mathcal{P}}(\Gamma) = \hbox{span}\{r(\cdot, p): p\in\mathcal{P}\}.$
		In particular, if we let $\Pi_{\mathcal{P}}$ denote the orthogonal projection onto $\mathcal{H}_{\mathcal{P}}(\Gamma)$, we obtain that, for $f \in \mathcal{H}(\Gamma)$,
		$$
		\Pi_{\mathcal{P}}(f)(\cdot) = \sum_{p \in \mathcal{P}} c_p(f) r(\cdot, p),
		$$
		where $c_p(f)\in\mathbb{R}$. Hence, 
		$$r(\cdot,s) = \Pi_{\mathcal{P}}^\perp(r(\cdot,s)) + \sum_{p\in\mathcal{P}} c_p(s) r(\cdot, p).$$
		So that, by applying $\Phi^{-1}$, we obtain
		\begin{equation}\label{eq:decomp_X_part1}
			u(s) = u_{\mathcal{P}}(s) + \sum_{p\in\mathcal{P}} c_p(s) u(p),
		\end{equation}
		where $u_{\mathcal{P}}(s) := \Phi^{-1}(\Pi_{\mathcal{P}}^\perp(r(\cdot,s)))$, $u_{\mathcal{P}}(\cdot)$ is independent of $\sigma(u(p):p\in\mathcal{P})$, and, by construction, $u_{\mathcal{P}}(\cdot)$ is a centered Gaussian field on $\Gamma$ with Cameron-Martin space $\mathcal{H}_{\mathcal{P}}(\Gamma)^\perp$. Now, observe that \eqref{eq:decomp_X_part1} gives us that
		$$u_{\mathcal{P}}(\cdot) \stackrel{d}{=} u(\cdot) | u(p) = 0: p\in\mathcal{P}.$$
		Furthermore, since $u$ is a GMRF on $\Gamma$, it follows from Theorem 1 in the main text that the space $\mathcal{H}(\Gamma)$ is local. Now, observe that $\mathcal{H}_{\mathcal{P}}(\Gamma)^\perp\subset \mathcal{H}(\Gamma)$, so that $\mathcal{H}_{\mathcal{P}}(\Gamma)^\perp$ is local, which implies that $u_{\mathcal{P}}(\cdot)$ is a GMRF.

		Now, observe that $u_{\mathcal{P}}(\cdot)$ is independent of $\sigma(u(p):p\in\mathcal{P})$. Thus, for $p\in\mathcal{P}$,
		$$\mathbb{E}(u_{\mathcal{P}}(p)^2) = \mathbb{E}(u_{\mathcal{P}}(p)^2 + u_{\mathcal{P}}(p)(u(p) - u_{\mathcal{P}}(p))) = \mathbb{E}(u_{\mathcal{P}}(p) u(p)) = 0,$$
		which shows that $u_{\mathcal{P}}(p) = 0$ almost surely for $p\in\mathcal{P}$. 
		
		It remains to be shown that $u_{\mathcal{P}}(\cdot)$ is Markov of order 1. To this end, observe that by \eqref{eq:decomp_X_part1}, we have that for any open set $O\subset\Gamma$ whose boundary consists of finitely many points,
		$$H_{\mathcal{P},+}^\perp(\partial O) \oplus H_{\mathcal{P},+}(\partial O) = H_+(\partial O),$$
		where, for a set $S$, $H^\perp_\mathcal{P}(S)$ denotes the Gaussian space associated to $u_{\mathcal{P}}(\cdot)$ on the set $S$.
		Therefore, since $u(\cdot)$ is Markov of order 1, it follows from Proposition \ref{prp:charmarkovspaces}, Definition \ref{def:MarkovPropertyFieldOrderP} and \cite[Lemma 3.3]{Mandrekar1976} that 
		$$H_+(\partial O) = \hbox{span}\{u(s):s\in\partial O\}.$$
		Therefore, 
		\begin{equation}\label{eq:identity_space_1}
			H_{\mathcal{P},+}^\perp(\partial O) = H_+(\partial O) \ominus H_{\mathcal{P},+}(\partial O) = \hbox{span}\{u(s):s\in\partial O\} \ominus H_{\mathcal{P}}(\partial O).
		\end{equation}
		Now, let $\widehat{\Pi}_{\mathcal{P}}^\perp$ be the projection from $H(\Gamma)$ onto $H_{\mathcal{P}}(\Gamma)^\perp$ and observe that \eqref{eq:identity_space_1} implies that 
		\begin{align*}
			H_{\mathcal{P},+}^\perp(\partial O) &\supset \hbox{span}\{u(s):s\in\partial O\} \ominus H_{\mathcal{P}}(\Gamma) = \hbox{span}\{\widehat{\Pi}_{\mathcal{P}}^\perp(u(s)):s\in\partial O\}\\
			&= \hbox{span}\{u_{\mathcal{P}}(s):s\in\partial O\}.
		\end{align*}
		On the other hand, we have that $H_{\mathcal{P},+}^\perp(\partial O) \subset \hbox{span}\{u(s):s\in\partial O\}$, which implies that
		\begin{align*}
			H_{\mathcal{P},+}^\perp(\partial O) &= \widehat{\Pi}_{\mathcal{P}}^\perp(H_{\mathcal{P},+}^\perp(\partial O)) \subset \widehat{\Pi}_{\mathcal{P}}^\perp(\hbox{span}\{u(s):s\in\partial O\})\\
			& = \hbox{span}\{u_{\mathcal{P}}(s):s\in\partial O\}.
		\end{align*}
		Therefore, we obtain that
		$H_{\mathcal{P},+}^\perp(\partial O) = \hbox{span}\{u_{\mathcal{P}}(s):s\in\partial O\}.$ We can now apply \cite[Lemma 3.3]{Mandrekar1976} to obtain that 
		$$
		\mathcal{F}_{+,\mathcal{P}}^u(\partial{O}) = \sigma(H_{\mathcal{P},+}^\perp(\partial O)) = \sigma(\hbox{span}\{u_{\mathcal{P}}(s):s\in\partial O\})
		$$
		 for every $O\subset\Gamma$, where 
		 $$
		 \mathcal{F}_{+,\mathcal{P}}^u(\partial{O}) = \bigcap_{\varepsilon>0} \mathcal{F}^u_{\mathcal{P}}((\partial O)_{\varepsilon}), 
		 $$
		 and for a set $S$, $\mathcal{F}_{\mathcal{P}}^u(S) = \sigma(u_{\mathcal{P}}(s): s\in S)$. This shows that  $u_{\mathcal{P}}(\cdot)$ is Markov of order 1 and concludes the proof.
		\end{proof}

		Next, we obtain a useful identity for covariance functions of GRFs on metric graphs when a single point splits the graph into two parts:

        \begin{lemma}\label{lem:borisov_type}
            Let $u(\cdot)$ be a GRF on a compact metric graph $\Gamma$. Let $e\in\mathcal{E}$ be an edge of $\Gamma$ and assume that for $t\in e$, $\sigma(u(t))$ splits $\sigma(u(s): s\in O)$ and $\sigma(u(s): s\in O^c)$, where $O\subset\Gamma$ is an open set such that $(w,e)\in O$ for $w>t$, and $(w,e)\not\in O$ for $w\leq t$, where $e = \{(w,e): w \in [0,l_e]\}$, $l_e>0$. Therefore, if $r(\cdot,\cdot)$ denotes the covariance function of $u$, 
            \begin{equation}\label{eq:borisov}
                r(t,t)r(a,b) = r(a,t)r(b,t),
            \end{equation}
            where $a\in O$ and $b\in O^c$. 
        \end{lemma}

        \begin{proof}
			Let us assume that ${\mathbb{E}(u(t)^2)>0}$, since if $u(t) = 0$ almost surely, \eqref{eq:borisov} is trivial. Let $O$ be as in the statement. Since $\sigma(u(t))$ splits $\sigma(u(s): s\in O)$ and $\sigma(u(s): s\in O^c)$, it follows that for $a\in O$ and $b\in O^c$, $u(a)$ and $u(b)$ are conditionally independent given $u(t)$:
$$\mathbb{E}(u(a)u(b)|u(t)) = E(u(a)|u(t)) E(u(b)|u(t)).$$
Now, since $u(\cdot)$ is a Gaussian field, it follows that
$$\mathbb{E}(u(a)|u(t)) = \frac{\mathbb{E}(u(a)u(t))}{\mathbb{E}(u(t)^2)} u(t) \quad\hbox{and}\quad \mathbb{E}(u(b)|u(t)) = \frac{\mathbb{E}(u(b)u(t))}{\mathbb{E}(u(t)^2)} u(t).$$
Thus,
$$\mathbb{E}(u(a)u(b)|u(t)) = \frac{\mathbb{E}(u(a)u(t))\mathbb{E}(u(b)u(t))}{\mathbb{E}(u(t)^2)^2} u(t)^2.$$
The result now follows from taking expectation on both sides of the above expression.
		\end{proof}

		Observe that \eqref{eq:borisov} is reminiscent of (1) in \cite{borisov_markov}. 
		The final result we need to prove the main result is that the identity \eqref{eq:borisov} implies some non-degeneracy of the covariance function.

        \begin{lemma}\label{lem:cov_func_nondeg}
            Let $\Gamma$ be a compact metric graph and $d(\cdot,\cdot)$ be a metric that satisfies the intermediate value property on $\Gamma$, is regular and is connected. Let $r$ be a continuous covariance function on a compact metric graph $\Gamma$. Additionally, suppose that given an edge $e\in\mathcal{E}$ and any $t_1,t_2,t_3\in e$, we have that
            \begin{equation}\label{eq:borisov2}
                r(t_1,t_2) r(t_2,t_3) = r(t_2,t_2) r(t_1,t_3).
            \end{equation}
            If for every $t\in e$, we have $r(t,t)>0$, then for any $t,s\in e$, we have $r(t,s)>0$.
        \end{lemma}

\begin{proof}
	We begin by showing that for every $t,s\in e$, we have $r(t,s)\neq 0$. We will prove this by contradiction. To this end, suppose that there exist $t,s\in e$ such that $r(t,s) = 0$. Thus, by \eqref{eq:borisov2}, for every $t^*\in e$, either $r(t,t^*) =0$ or $r(s,t^*)=0$. 

Observe that, by assumption, for every $w\in e$, $r(w,w)>0$. Since $e\times e$ is compact and $r$ is continuous, we have that $r$ is uniformly continuous on $e\times e$. Therefore, there exists $\delta>0$ such that for every $u,v \in e$, satisfying $d(u,v)<\delta$, we have $r(u,v)>0$. 

Let $[t,s]$ denote the path inside the edge $e$ that connects $t$ to $s$. Take any $\delta^*$ such that $0<\delta^* < \min\{\delta, d(t,s)\}$, and let $p(t,s;1)\in [t,s]$ be such that $d(s, p(t,s;1)) = \delta^*$. By \eqref{eq:borisov2}, since $r(s, p(t,s;1)) > 0$, we have that $r(t, p(t,s;1)) = 0$. We now repeat the argument by choosing $p(t,s;2)\in [t, p(t,s;1)]$ such that $d(p(t,s;1), p(t,s;2)) = \delta^*$. By \eqref{eq:borisov2}, we have that $r(t, p(t,s;2)) = 0$, since $r(p(t,s;1), p(t,s;2)) > 0$. Hence, by proceeding like this, we obtain a sequence of points $p(t,s;k) \in [t, p(t,s;k-1)]$ such that $r(t, p(t,s;k)) = 0$, and $d(p(t,s;k), p(t,s;k-1)) = \delta^*$. Thus, it follows from the definition of regular metric, that there exists some $N\in\mathbb{N}$ such that $d(t, p(t,s;N)) < \delta^*$. This gives us a contradiction since this implies that $r(t,p(t,s;N)) > 0$. 

This contradiction shows that for every $t,s\in e$, we have $r(t,s)\neq 0$. The result now follows from continuity of $r$ and the intermediate value theorem. Indeed, if there exist $t,s\in e$ such that $r(t,s) < 0$, by the intermediate value theorem and the fact that the metric is connected, there exists some $t^* \in [t,s]$ such that $r(t,t^*)=0$, which yields a contradiction.
\end{proof}

We are now ready to prove the main result.

\begin{proof}[Proof of Theorem~\ref{thm:explicit_link}]
	Let $u(\mathcal{P}) := \{u(s): s\in \mathcal{P}\}$. Then, we have that ${u(\mathcal{P}) \sim \mathcal{N}(\boldsymbol{0}, \boldsymbol{\Sigma})}$, where $\boldsymbol{\Sigma} = \{r(t,s): t,s\in \mathcal{P}\}$. Since $r$ is strictly positive definite, the precision matrix $\boldsymbol{Q} := \boldsymbol{\Sigma}^{-1} = [Q_{ts}]_{t,s\in\mathcal{P}}$ exists. Furthermore, the strict positive-definiteness also implies that for every $s\in \mathcal{P}, Q_{ss} > 0$. 

	Let $s,t\in\mathcal{P}$ be such that $t\not\stackrel{\mathcal{E}_{\mathcal{P}}}{\sim} s$. Then, since $u(\cdot)$ is a GMRF of order 1, 
	$$
	u(t) \perp u(s) | \{u(w): w\in \mathcal{P}\setminus\{t,s\}\}. 
	$$
	Indeed, take the open set 
	$$
	O = \hbox{int } \mathcal{E}_{t,\mathcal{P}} \cup \{t\} \cup (\hbox{int }\mathcal{E}_{\mathcal{P}}\setminus \mathcal{E}_{s,\mathcal{P}}), 
	$$
	where 
	$$
	\hbox{int }\mathcal{E}_{\mathcal{P}} = \bigcup_{e\in \mathcal{E}_{\mathcal{P}}} \hbox{int }e
	$$ and, for $p\in\mathcal{P}$, $\mathcal{E}_{p,\mathcal{P}}$ denotes the set of all edges in $\mathcal{E}_{\mathcal{P}}$ that are incident to $p$. Note that $\hbox{int } \mathcal{E}_{t,\mathcal{P}} \cup \{t\}$ is open in $\Gamma$ and $\hbox{int }\mathcal{E}_{\mathcal{P}}\setminus \mathcal{E}_{s,\mathcal{P}}$ is open. Further, $\partial O = \mathcal{P}\setminus \{t,s\}$, with $t\in O$ and $s\not\in O$. Therefore, if $t\not\stackrel{\mathcal{E}_{\mathcal{P}}}{\sim} s$, we have $Q_{st} = 0$. 

	It remains to be shown that if $t\stackrel{\mathcal{E}_{\mathcal{P}}}{\sim} s$, then $Q_{st} < 0$. To this end, let $t,s\in\mathcal{P}$ be such that $t\stackrel{\mathcal{E}_{\mathcal{P}}}{\sim} s$. We start by considering the conditioned field
	$$
	u_{t,s}(\cdot) = u(\cdot) | u(p) = 0: p\in\mathcal{P}\setminus\{t,s\},
	$$
	which, by Lemma~\ref{lem:condpoints_markov}, is a GMRF of order 1. Let $f = [t,s]\in \mathcal{E}_{\mathcal{P}}$, and note that there exists $e\in\mathcal{E}$ such that $f\subset e$. Take any $t^*\in f$ and let 
	$$
	O = (t^*, s]\cup (\hbox{int }\mathcal{E}_{s,\mathcal{P}}) \cup (\hbox{int }\mathcal{E}_{\mathcal{P}}\setminus \mathcal{E}_{t,\mathcal{P}}). 
	$$
	Because we removed all multiple edges by considering only admissible points, we have that $\partial O = \{t^*\} \cup (\mathcal{P}\setminus\{t,s\})$. Furthermore, since $u_{t,s}(\cdot)$ is Markov of order 1, we have that 
	$$
	\sigma(u_{t,s}(w): w\in \{t^*\} \cup (\mathcal{P}\setminus\{t,s\}))
	$$ 
	splits $\sigma(X_{t,s}(w): w\in O)$ and $\sigma(X_{t,s}(w): w\in O^c)$. Observe that by Lemma~\ref{lem:condpoints_markov}, $u_{t,s}(w) = 0$ almost surely for $w\in \mathcal{P}\setminus \{t,s\}$. Therefore,
	$$
	\sigma(u_{t,s}(w): w\in \{t^*\} \cup (\mathcal{P}\setminus\{t,s\})) = \sigma(u_{t,s}(t^*)),
	$$
	whence, $u_{t,s}(t^*)$ splits $\sigma(u_{t,s}(w): w\in O)$ and $\sigma(u_{t,s}(w): w\in O^c)$. Thus, if $r_c(\cdot,\cdot)$ denotes the covariance function of $u_{t,s}(\cdot)$, we obtain from Lemma~\ref{lem:borisov_type} that for every $a\in O$ and $b\in O^c$
	$$
	r_c(t^*,t^*) r_c(a,b) = r_c(a,t^*) r_c(b,t^*).
	$$
	Further, since $t^*\in [t,s]$ was arbitrary, it follows that for every $t_1,t_2,t_3\in [t,s]$, we have
	$$
	r_c(t_1,t_2) r_c(t_2,t_3) = r_c(t_2,t_2) r_c(t_1,t_3).
	$$
	Therefore, by Lemma~\ref{lem:cov_func_nondeg}, we have that $r_c(t,s) > 0$. On the other hand, observe that
	\begin{align*}
		u(\{t,s\}) | u(\mathcal{P}\setminus \{t,s\}) &\sim \mathcal{N}\left(0, \begin{bmatrix}
			r_c(t,t) &  r_c(t,s) \\
			r_c(t,s) & r_c(s,s)
		   \end{bmatrix} \right) \\
		   &= \mathcal{N}\left(0, \frac{1}{Q_{tt}Q_{ss} - Q_{ts}^2} \begin{bmatrix}
			Q_{tt} & -Q_{ts} \\
			-Q_{ts} & Q_{ss}
			\end{bmatrix} \right).
	\end{align*}
	So, we have that
	$$
	0 < r_c(t,s) = \frac{-Q_{ts}}{Q_{tt}Q_{ss} - Q_{ts}^2} \Rightarrow Q_{ts} < 0.
	$$
	This concludes the proof.
\end{proof}

\section{Examples and consequences}\label{sec:illustration}
\subsection{Whittle--Mat\'ern fields on regular lattices}
As a specific example of a Gaussian Markov random field on a metric graph, consider the metric graph $\Gamma$ which is a regular lattice with $n$ vertices, where each edge is of length $\ell$. Let $u$ be a Whittle--Mat\'ern field, specified as the solution to
$$
(\kappa^2 - \Delta_{\Gamma})^{\alpha/2} (\tau u) = \mathcal{W}, \quad\text{on $\Gamma$}
$$
where $\mathcal{W}$ is Gaussian white noise and $\Delta_{\Gamma}$ denotes the Kirchhoff-Laplacian \cite{BSW2022}. Here, $\kappa, \tau >0$ and $\alpha>1/2$ are parameters which determine the marginal variances, practical correlation ranges, and sample path regularity of the solution. By \cite{BSW_Markov}, this is a centered GMRF of order 1 if $\alpha=1$, and by \cite[][Corollary 3]{BSW2023_AOS}, the process evaluated at the vertices is a random variable $x = (x_1, \ldots, x_n)^{\top}$ which has a centered multivariate Gaussian distribution with precision matrix $Q$, with non-zero elements 
\begin{align}
Q_{ij} &= 2\kappa \tau^2 \cdot \begin{cases}
d_i \left(\frac12 + \frac{e^{-2\kappa\ell}}{1-e^{-2\kappa\ell}}\right) & i=j\\
- \frac{e^{-\kappa\ell}}{1-e^{-2\kappa\ell}} & i\sim j
\end{cases}\notag \\
&= \frac{\kappa \tau^2}{\sinh(\kappa\ell)} \cdot \begin{cases}
	d_i \cosh(\kappa\ell) & i=j\\
	- 1 & i\sim j
	\end{cases}, \label{eq:Qwm}
\end{align}
where $d_i$ denotes the degree of vertex $i$. This means that we can view the model as a conditional autoregressive (CAR) model specified by 
$$
\mathbb{E}(x_i|x_{-i}) = - \sum_{j : j\sim i} \beta_{ij} x_j\quad\text{and}\quad
\text{Prec}(x_i|x_{-i}) = \kappa_i,
$$
where $x_{-i}$ denotes all components in $x$ except for $x_i$,  
$$
\beta_{ij} = \frac1{d_i\cosh(\kappa\ell)}, \quad\text{and}\quad \kappa_i = \frac{\kappa\tau^2 d_i}{\tanh(\kappa\ell)}.
$$
We can note that this conditional autoregression is different from a standard first-order CAR model that is often used in spatial statistics \cite[see, e.g.][]{lindgren11}, which would have 
\begin{equation}\label{eq:car}
	Q_{ij} 
	= \widetilde{\tau}^2 \cdot \begin{cases}
		a + d_i & i=j\\
		- 1 & i\sim j
		\end{cases}\\
	\end{equation}
	for some $a>0$. However, note that \eqref{eq:Qwm} corresponds to a proper CAR model as long as $\kappa\ell>0$ as we then have $\cosh(\kappa\ell)>1$, which means that $Q$ is diagonally dominant and therefore strictly positive definite \cite{rue2005gaussian}.
	Further, letting $\kappa\rightarrow 0$, the precision matrix \eqref{eq:Qwm} converges to a standard first-order intrinsic CAR model, obtained by setting $a=0$ in \eqref{eq:car}.

\subsection{Non-Markovianity of isotropic models}
An important consequence of the link between GMRFs of order 1 and graphical models is that, under a homogeneity assumption on the metric, there is an incompatibility between Markov property of order 1 and isotropy. This problem was addressed, and partially answered, in \cite[Theorem 3]{BSW_Markov}. In the present work, we extend these results by providing a complete characterization for all cases under the geodesic metric (see Remark \ref{rem:characterization_iso_markov_geo}) and by broadening the class of metric graphs covered for the resistance metric (see Remark \ref{rem:characterization_iso_markov_res}). To set the stage for these advancements, we first review the current state-of-the-art results in this area, which necessitates introducing some auxiliary definitions. 

We begin by reviewing the concept of the 1-sum of metric spaces, which naturally extends to the notion of the 1-sum of metric graphs.

\begin{definition}\label{def:1sum_spaces}
Let $(X_1, d_1)$ and $(X_2, d_2)$ be two metric spaces such that $X_1 \cap X_2 = \{x_0\}$. The 1-sum of $(X_1, d_1)$ and $(X_2, d_2)$, denoted by $(X_1 \cup X_2, d)$, is the metric space defined by:
\[
d(x, y) = 
\begin{cases} 
d_1(x, y), & \text{if } x, y \in X_1, \\
d_2(x, y), & \text{if } x, y \in X_2, \\
d_1(x, x_0) + d_2(x_0, y), & \text{if } x \in X_1 \text{ and } y \in X_2.
\end{cases}
\]
\end{definition}

By iterating this construction, we can define $k$-step 1-sums of metric graphs.

\begin{definition}\label{def:kstep_1sum}
Given an ordered collection of metric graphs $\Gamma_1, \ldots, \Gamma_k$, where $k \in \mathbb{N}$, suppose there exist points $v_2 \in \Gamma_2, \ldots, v_k \in \Gamma_k$ such that:
\[
\Gamma_1 \cap \Gamma_2 = \{v_1\}, \quad (\Gamma_1 \cup \Gamma_2) \cap \Gamma_3 = \{v_3\}, \quad \ldots, \quad (\Gamma_1 \cup \cdots \cup \Gamma_{k-1}) \cap \Gamma_k = \{v_k\}.
\]
Let $\widetilde{\Gamma}_2$ denote the 1-sum of $\Gamma_1$ and $\Gamma_2$, and define $\widetilde{\Gamma}_3$ as the 1-sum of $\widetilde{\Gamma}_2$ and $\Gamma_3$. This process is repeated iteratively, so that $\widetilde{\Gamma}_k$ is the 1-sum of $\widetilde{\Gamma}_{k-1}$ and $\Gamma_k$. We refer to $\widetilde{\Gamma}_k$ as the $k$-step 1-sum of $\Gamma_1, \ldots, \Gamma_k$, with the intersecting points $\{v_2, \ldots, v_k\}$.
\end{definition}

In \cite[Theorem 3]{BSW_Markov}, the following assumption was introduced, leading to the subsequent theorem:

\begin{assumption}\label{assump:iso_markov}
	Let $S$ be a Euclidean cycle, and let $k_1, k_2 \in \mathbb{N}$. Consider a collection of $k_1 + k_2 - 2$ metric graphs $\Gamma_1, \ldots, \Gamma_{k_1+k_2-2}$ with Euclidean edges, where $\Gamma_i \cap \Gamma_j = \emptyset$ for $i \neq j$. Furthermore, let $T$ be either a Euclidean cycle with a length distinct from that of $S$, or an edge. The metric graph $\Gamma$ is then defined as the $(k_1+k_2)$-step 1-sum of $S, \Gamma_1, \ldots, \Gamma_{k_1-1}, T, \Gamma_{k_1}, \ldots, \Gamma_{k_1+k_2-2}$.
\end{assumption}

\begin{theorem}[Theorem 3, \cite{BSW_Markov}]\label{thm:markov_resistance}
	Let $\Gamma$ be a metric graph satisfying Assumption~\ref{assump:iso_markov}. Suppose $X(\cdot)$ is a Gaussian random field (GRF) on $\Gamma$ with an isotropic covariance function $\rho(s,t) = r(\widetilde{d}(s,t))$, where $r(\cdot)$ is continuous and $\widetilde{d}(s,t)$ represents either the resistance metric or the geodesic metric. Then, $X$ is not Markov of order 1.
\end{theorem}

In this work we are able to prove this incompatibility result for any metric defined on a compact metric graph such that the following homogeneous condition holds:

\begin{definition}\label{def:homogeneous_metric}
	Let $\Gamma$ be a compact metric graph endowed with a metric $d(\cdot,\cdot)$ that satisfies the intermediate value property, is regular and connected in the sense of Definition~\ref{def:intermediatevalue}. Given two edges $e_1,e_2$ of $\Gamma$ such that $e_1\cap e_2 = \{v\}$, we say that the metric is homogeneous on $e_1$ and $e_2$  if there exist $t_1,t_2,t_3\in e_1$ and $s_1\in e_1, s_2,s_3\in e_2$ such that the distance matrix of $\{t_1,t_2,t_3,v\}$ coincides with the distance matrix of $\{s_1,v,s_2,s_3\}$. 
\end{definition}

\begin{remark}
	The geodesic metric is homogeneous for any pair of edges $e_1, e_2$ such that $e_1 \cap e_2 = \{v\}$ on a compact metric graph. Further, let $S = S_1 \cup S_2$ denote the 1-sum of two Euclidean cycles. Consider a collection of $k - 1$ metric graphs $\Gamma_1, \ldots, \Gamma_{k-1}$ with Euclidean edges, where $\Gamma_i \cap \Gamma_j = \emptyset$ for $i \neq j$, and let $\Gamma$ be the $k$-step 1-sum of the metric graphs $S, \Gamma_1, \ldots, \Gamma_{k-1}$. Then, given edges $e_1 \in S_1$ and $e_2 \in S_2$ such that $e_1 \cap e_2 = \{v\}$, the resistance metric is also homogeneous for $e_1$ and $e_2$ on the metric graph $\Gamma$.
\end{remark}

The following corollary is the main result of this section:

\begin{corollary}\label{cor:incomp_iso_markov}
	Let $\Gamma$ be a compact and connected metric graph that properly contains a cycle (i.e., $\Gamma$ contains a cycle but is not a cycle). Let $e_1, e_2 \in \mathcal{E}$ be such that $e_2$ is in a cycle whereas $e_1$ is not on the cycle that contains $e_2$ and $e_1\cap e_2 = \{v\}$. Suppose that the metric $d(\cdot,\cdot)$ is homogeneous on $e_1$ and $e_2$. Let, now, $u$ be a GRF on $\Gamma$ with a continuous and strictly positive definite covariance function. Then, $u$ cannot be simultaneously isotropic and Markov of order~1.
\end{corollary}

\begin{proof}
	Assume, by contradiction, that $u$ is both isotropic and Markov of order 1. Let $v, e_1$ and $e_2$ be as in the statement of this corollary. By homogeneity, there exist $t_1,t_2,t_3\in e_1$ and $s_1\in e_1, s_2,s_3\in e_2$ such that the distance matrices of $\{t_1,t_2,t_3,v\}$ and $\{s_1,v,s_2,s_3\}$ coincide. Assume, without loss of generality, that these points are ordered in $e_1$ as ${t_1 < t_2 < t_3 < v}$, and on $e_2$ as $v < s_2 < s_3$. 
	Now, since $u$ is a GMRF of order 1, we have that $u(t_2)\perp u(v) | \{u(t_1), u(t_3)\}$. By isotropy and Gaussianity, it follows that $u(v) \perp u(s_3) | \{u(s_1), u(s_2)\}$. However, in view of Corollary~\ref{cor:faithfulness}, this is a contradiction since $\{s_1,s_2\}$ does not separate $v$ and $s_3$, as there is a path along the cycle connecting them. More precisely, Corollary \ref{cor:faithfulness} implies that $\{u(p): p\in\mathcal{P}\}$, where $\mathcal{P}$ is any admissible set of points containing $s_1,s_2,s_3$ is faithful to its pairwise independence graph, which in turns implies that $u(v)$ and $u(s_3)$ must be conditionally dependent given $u(s_1)$ and $u(s_2)$, since $\{s_1,s_2\}$ does not separate $\{v,s_3\}$. This contradiction proves the result.
\end{proof}

The following remark summarizes the impact of Corollary \ref{cor:incomp_iso_markov} on compact metric graphs endowed with geodesic metrics:

\begin{remark}\label{rem:characterization_iso_markov_geo}
	Observe that if $\Gamma$ is a compact metric graph endowed with the geodesic metric, then $\Gamma$ is not covered by Corollary \ref{cor:incomp_iso_markov} if, and only if, $\Gamma$ is either a tree or a cycle. However, it is known from \cite[Theorem 2]{BSW_Markov}, that if $\Gamma$ is a cycle, there exists a Gaussian Markov random field of order 1 with isotropic covariance function, and it is known from \cite{moller2022lgcp} that a Gaussian field on trees with isotropic exponential covariance function is Markov of order 1. Therefore, Corollary \ref{cor:incomp_iso_markov} is sharp for the geodesic metric and cannot be further improved. Thus providing a major extension of \cite[Theorem 3]{BSW_Markov} in the case of geodesic metric. Examples of graphs where no fields exists which are both Markov and isotropic in the geodesic metric are given in Figure~\ref{fig:graphs}. These examples are also examples of graphs which were not covered in \citep[][Theorem~3]{BSW_Markov}.
\end{remark}

With regard to the resistance metric, we have the following observation:

\begin{remark}\label{rem:characterization_iso_markov_res}
	If $\Gamma$ is a compact metric graph equipped with the resistance metric, it is more difficult to verify the homogeneity property required in Corollary \ref{cor:incomp_iso_markov}. However, the corollary at least allows us to remove the constraint in Theorem \ref{thm:markov_resistance} that the Euclidean cycles must have different lengths.
\end{remark}

\begin{figure}
	\definecolor{nodecolor}{rgb}{0.8,0.8,0.8}
  \centering
  \begin{center}
	\begin{tikzpicture}[scale=0.5] 
		\draw[thick] (0,0) circle (2);
		\draw[thick] (-4,0) circle (2);
		\draw [fill = nodecolor] (-6,0) circle (8pt);
		\draw [fill = nodecolor] (-2,0) circle (8pt);

		\draw[thick] (4,2) -- (8,2);
		\draw[thick] (4,-2) -- (8,-2);
		\draw[thick] (4,2) -- (4,-2);
		\draw[thick] (8,-2) -- (8,2);
		\draw[thick] (4,-2) -- (8,2);
		\draw [fill = nodecolor] (4,2) circle (8pt);
		\draw [fill = nodecolor] (8,2) circle (8pt);
		\draw [fill = nodecolor] (4,-2) circle (8pt);
		\draw [fill = nodecolor] (8,-2) circle (8pt);

		\draw[thick] (10,2) -- (16,2);
		\draw[thick] (10,-2) -- (16,-2);
		\draw[thick] (10,0) -- (16,0);
		\draw[thick] (10,2) -- (10,-2);
		\draw[thick] (12,-2) -- (12,2);
		\draw[thick] (14,-2) -- (14,2);
		\draw[thick] (16,-2) -- (16,2);
		\draw [fill = nodecolor] (10,2) circle (8pt);
		\draw [fill = nodecolor] (12,2) circle (8pt);
		\draw [fill = nodecolor] (14,2) circle (8pt);
		\draw [fill = nodecolor] (16,2) circle (8pt);
		\draw [fill = nodecolor] (10,0) circle (8pt);
		\draw [fill = nodecolor] (12,0) circle (8pt);
		\draw [fill = nodecolor] (14,0) circle (8pt);
		\draw [fill = nodecolor] (16,0) circle (8pt);
		\draw [fill = nodecolor] (10,-2) circle (8pt);
		\draw [fill = nodecolor] (12,-2) circle (8pt);
		\draw [fill = nodecolor] (14,-2) circle (8pt);
		\draw [fill = nodecolor] (16,-2) circle (8pt);
	\end{tikzpicture}
\end{center}
	\caption{Examples of graphs where no random fields exist which are both Markov of order 1 and isotropic for the geodesic metric. 
	}
	\label{fig:graphs}
	\end{figure}
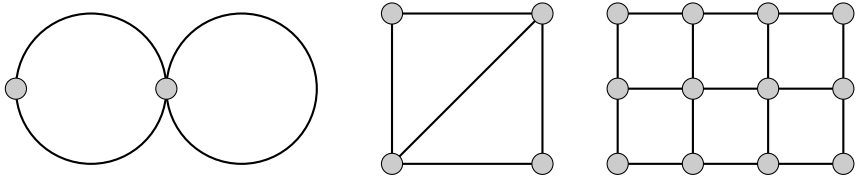

	To illustrate this result, consider a metric graph consisting of two cycles with 7 vertices, where one vertex is shared by both cycles (see Figure \ref{fig:tadpole}), and denote this graph by $\Gamma$. Unlike the leftmost metric graph in Figure \ref{fig:graphs}, $\Gamma$ has Euclidean edges, as defined in \cite{anderes2020isotropic}. Note also that this graph does not satisfy Assumption \ref{assump:iso_markov}, which is necessary for the application of \cite[Theorem~3]{BSW_Markov}. Therefore, we will examine a Gaussian field with an exponential isotropic covariance function defined either in terms of the resistance metric or the geodesic metric on $\Gamma$.

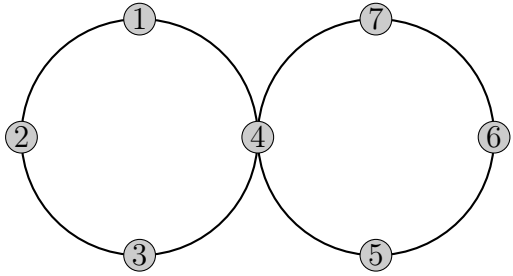
\begin{figure}
	\definecolor{nodecolor}{rgb}{0.8,0.8,0.8}
  \centering
  \begin{center}
	\begin{tikzpicture}[scale=0.5] 
		\draw[thick] (pi,0) circle (pi);
		\draw[thick] (-pi,0) circle (pi);
		\draw [fill = nodecolor] (-pi,pi) circle (12pt);
		\node[align=center] at (-pi,pi) {1};
		\draw [fill = nodecolor] (-pi,-pi) circle (12pt);
		\node[align=center] at (-pi,-pi) {3};
		\draw [fill = nodecolor] (0,0) circle (12pt);
		\node[align=center] at (0,0) {4};
		\draw [fill = nodecolor] (-2*pi,0) circle (12pt);
		\node[align=center] at (-2*pi,0) {2};
		\draw [fill = nodecolor] (pi,pi) circle (12pt);
		\node[align=center] at (pi,pi) {7};
		\draw [fill = nodecolor] (pi,-pi) circle (12pt);
		\node[align=center] at (pi,-pi) {5};
		\draw [fill = nodecolor] (2*pi,0) circle (12pt);
		\node[align=center] at (2*pi,0) {6};
	\end{tikzpicture}
\end{center}
	\caption{A 2-cycle graph with 7 vertices and 8 edges, each of length 1. This graph also represents the pairwise independence graph of the vertex process for the Whittle--Matérn model with $\alpha=1$ defined on this metric graph.}
	\label{fig:tadpole}
	\end{figure}

	Specifically, let $d_R(\cdot,\cdot)$ be the resistance metric on $\Gamma$, let $d(\cdot, \cdot)$ be the geodesic metric, and let $\rho_{\exp}(\cdot)$ be the exponential covariance function given by
	\begin{equation}\label{eq:expcov}
	\rho_{\exp}(h) = \sigma^2 \exp\{-\kappa h\},
	\end{equation}
	where $\kappa>0$ is related to the correlation range of the process and $\sigma>0$ is the (marginal) standard deviation. Thus, the first field we will consider on $\Gamma$ is $u_{\exp}$ given by a Gaussian field with covariance function $\varrho_{res}(s_1,s_2) = \rho_{\exp}(d_R(s_1,s_2))$ or $\varrho_{geo}(s_1,s_2) = \rho_{\exp}(d(s_1,s_2))$. 

	If this process was markov with the dependence structure given by the graph in Figure~\ref{fig:tadpole}, 
	the precision matrix of the process evaluated in the vertices should have the following non-zero structure
	$$
	\begin{bmatrix}
	x & x & 0 & x & 0 & 0 & 0\\
	x & x & x & 0 & 0 & 0 & 0\\
	0 & x & x & x & 0 & 0 & 0\\
	x & 0 & x & x & x & 0 & x\\
	0 & 0 & 0 & x & x & x & 0\\
	0 & 0 & 0 & 0 & x & x & x\\
	0 & 0 & 0 & x & 0 & x & x\\
	\end{bmatrix}
	$$
	where \( x \) denotes a non-zero element. By evaluating the precision matrix corresponding to \eqref{eq:expcov} with \(\sigma > 0\) and \(\kappa > 0\), and applying either the resistance metric or the geodesic metric, we obtain (using Mathematica \cite{Mathematica}) a precision matrix of the form 
$$
Q = \frac{1}{\sigma^2}
\begin{bmatrix}
q_1 & q_2 & q_3 & q_2 & 0 & 0 & 0 \\
q_2 & q_1 & q_2 & q_3 & 0 & 0 & 0 \\
q_3 & q_2 & q_1 & q_2 & 0 & 0 & 0 \\
q_2 & q_3 & q_2 & q_4 & q_2 & q_3 & q_2 \\
0 & 0 & 0 & q_2 & q_1 & q_2 & q_3 \\
0 & 0 & 0 & q_3 & q_2 & q_1 & q_2 \\
0 & 0 & 0 & q_2 & q_3 & q_2 & q_1 \\
\end{bmatrix},
$$
where, for the resistance metric,
\begin{align*}
q_1 &= \frac{e^{3 \kappa /2} \left(e^{\kappa /2} + e^{\kappa} + 2\right)}{\left(-3 e^{\kappa /2} - 2 e^{\kappa} + 2 e^{3 \kappa /2} + e^{2 \kappa} + e^{5 \kappa /2} + 1\right) },\\
q_2 &= -\frac{e^{5 \kappa /4}}{\left(-4 e^{\kappa /2} + 2 e^{\kappa} + e^{2 \kappa} + 1\right) },\\
q_3 &= -\frac{e^{\kappa}}{\left(2 e^{\kappa /2} + 4 e^{\kappa} + 2 e^{3 \kappa /2} + e^{2 \kappa} - 1\right) },\\
q_4 &= \frac{3 e^{\kappa /2} + 2 e^{\kappa} + 2 e^{3 \kappa /2} + e^{2 \kappa} + e^{5 \kappa /2} - 1}{\left(-3 e^{\kappa /2} - 2 e^{\kappa} + 2 e^{3 \kappa /2} + e^{2 \kappa} + e^{5 \kappa /2} + 1\right) }.
\end{align*}
and for the geodesic metric, 
$$
\begin{array}{ll}
    q_1 &= \dfrac{e^{4 \kappa }}{\left(e^{2 \kappa }-1\right)^2 },
    q_2 = -\dfrac{e^{3 \kappa }}{\left(e^{2 \kappa }-1\right)^2 }, \\[15pt]
    q_3 &= \dfrac{e^{2 \kappa }}{\left(e^{2 \kappa }-1\right)^2 }, 
    q_4 = \dfrac{2 e^{2 \kappa } + e^{4 \kappa } - 1}{\left(e^{2 \kappa }-1\right)^2 }.
\end{array}
	  $$
	In both cases, as $q_3\neq 0$, the model has a sparsity structure given by the graph in Figure~\ref{fig:graph2}, which is different from the original graph and thus confirms the theoretical results in that it does not have the Markov property. Thus, while the isotropic models exhibit some degree of conditional independence in this example, they are not Markov random fields of order 1.
	
	\begin{figure}
		\definecolor{nodecolor}{rgb}{0.8,0.8,0.8}
	  \centering
	  \begin{center}
		\begin{tikzpicture}[scale=0.5] 
			\draw[thick] (-3*pi,0) -- (-pi,0)  {};
			\draw[thick] (-pi,0) -- (0,pi)  {};
			\draw[thick] (-pi,0) -- (0,-pi)  {};
			\draw[thick] (0,pi) -- (pi,0)  {};
			\draw[thick] (0,-pi) -- (pi,0)  {};
			\draw[thick] (0,-pi) -- (0,-0.3)  {};
			\draw[thick] (0,0.3) -- (0,pi)  {};
			\draw[thick] (-pi,0) -- (pi,0)  {};
			\draw[thick] (-2*pi,pi) -- (-pi,0)  {};
			\draw[thick] (-2*pi,pi) -- (-3*pi,0)  {};
			\draw[thick] (-2*pi,-pi) -- (-pi,0)  {};
			\draw[thick] (-2*pi,-pi) -- (-3*pi,0)  {};
			\draw[thick, bend left=90] (0,-0.3) to (0,0.3)  {};
			\draw[thick] (-2*pi,pi) -- (-2*pi,0.3)  {};
			\draw[thick] (-2*pi,-pi) -- (-2*pi,-0.3)  {};
			\draw[thick, bend left=90] (-2*pi,-0.3) to (-2*pi,0.3)  {};
			\draw [fill = nodecolor] (-3*pi,0) circle (12pt);
			\node[align=center] at (-3*pi,0) {2};
			\draw [fill = nodecolor] (-pi,0) circle (12pt);
			\node[align=center] at (-pi,0) {4};
			\draw [fill = nodecolor] (0,pi) circle (12pt);
			\node[align=center] at (0,pi) {7};
			\draw [fill = nodecolor] (0,-pi) circle (12pt);
			\node[align=center] at (0, -pi) {5};
			\draw [fill = nodecolor] (pi,0) circle (12pt);
			\node[align=center] at (pi,0) {6};
			\draw [fill = nodecolor] (-2*pi,pi) circle (12pt);
			\node[align=center] at (-2*pi,pi) {1};
			\draw [fill = nodecolor] (-2*pi,-pi) circle (12pt);
			\node[align=center] at (-2*pi,-pi) {3};
		\end{tikzpicture}
	\end{center}
		\caption{The pairwise independence graph of the isotropic models on the 2-cycle graph.}
		\label{fig:graph2}
		\end{figure}
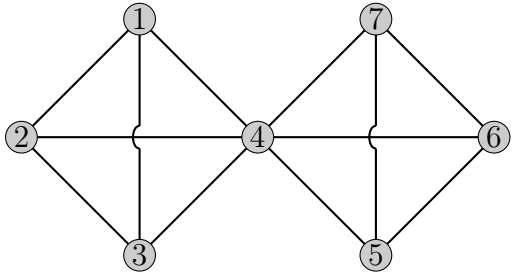	
	
	On the other hand, considering the distribution of the Whittle--Mat\'ern field with $\alpha = 1$, we obtain the precision matrix
	$$
	Q = \frac{\kappa\tau^2}{\sinh(\kappa)}	\begin{bmatrix}
		2b_\kappa & -1 & 0 & -1 & 0 & 0 & 0\\
		-1 & 2b_\kappa & -1 & 0 & 0 & 0 & 0\\
		0 & -1 & 2b_\kappa & -1 & 0 & 0 & 0\\
		-1 & 0 & -1 & 4b_\kappa & -1 & 0 & -1\\
		0 & 0 & 0 & -1 & 2b_\kappa & -1 & 0\\
		0 & 0 & 0 & 0 & -1 & 2b_\kappa & -1\\
		0 & 0 & 0 & -1 & 0 & -1 & 2b_\kappa\\
		\end{bmatrix},
	$$
	where $b_\kappa = \cosh(\kappa)$, 
	which thus has the correct sparsity structure for any values of $\kappa,\tau>0$. Additionally, this example illustrates the result stated in Corollary \ref{cor:faithfulness}. Specifically, the graph shown in Figure \ref{fig:tadpole} represents the pairwise independence graph of the vertex process for the Whittle--Matérn field with $\alpha = 1$.



\section{Discussion}\label{sec:discussion}

Although the results from this paper were established for first-order GMRFs, we believe that our findings can be extended to higher-order GMRFs, specifically those of order \(\alpha\), where \(\alpha \in \mathbb{N}\). The key idea is to consider the vector-valued field \(s \mapsto (u(s), \ldots, u^{(\alpha-1)}(s))\), which itself constitutes a first-order GMRF. While the joint covariance of all elements in this vector-valued field may not generally satisfy MTP\(_2\), we believe that the supporting results regarding MTP\(_2\) distributions and faithfulness could be extended to this more general case. In particular, this extension would involve adapting the framework to account for the precision matrix entries, which are now matrices induced by the vector-valued field \((u(s), \ldots, u^{(\alpha-1)}(s))\). This approach could provide a promising avenue for further research, allowing for a broader application of the concepts discussed in this work.

Given this extension, it might be tempting to also try to extend the result about the incompatibility between Markov assumptions and isotropy to higher order Markov fields. However, it is not even clear if it is possible to define isotropic and differentiable random fields on general metric graphs, because the construction of \citet{anderes2020isotropic} relied on complete monotonicity of the covariance function. Thus, in order to consider this extension, one would first need to show that there exists isotropic and differentiable random fields on general metric graphs. Furthermore, we conjecture that if $\Gamma$ is a compact metric graph endowed with the resistance metric such that the homogeneity condition in Definition \ref{def:homogeneous_metric} is violated, then, no Gaussian random field defined on $\Gamma$ can be simultaneously Markov of order 1 and isotropic. One should also observe that Euclidean cycles and trees, for which such fields exist, satisfy the homogeneity condition. 

Another promising direction of research lies in the application of faithfulness for simplifying metric graphs in statistical tasks such as estimation and prediction of GMRFs of order 1 on compact metric graphs. Specifically, when the data of interest is restricted to a smaller subgraph, faithfulness can help reduce the computational cost. Consider a compact metric graph $\Gamma$ with vertex set $\mathcal{V}$ and a GMRF $u$ of order 1 defined on $\Gamma$, which satisfies the conditions of Theorem \ref{thm:explicit_link}. For a subgraph $\widetilde{\Gamma} \subset \Gamma$, Theorem \ref{thm:explicit_link} guarantees that the process $\{u(s): s \in \widetilde{\Gamma}\}$ is conditionally independent of $\{u(s): s \in \Gamma \setminus \widetilde{\Gamma}\}$ given the values $\{u(s): s \in \mathcal{V} \cap \partial \widetilde{\Gamma}\}$. In practical terms, if the focus is on the subgraph $\widetilde{\Gamma}$ and information is available on $\mathcal{V} \cap \partial \widetilde{\Gamma}$, there is no need to study the full process on $\Gamma$, only the study on the smaller graph $\widetilde{\Gamma}$ is necessary. This reduction from $\Gamma$ to $\widetilde{\Gamma}$ can significantly lower the computational cost of statistical analysis, especially for large graphs. Moreover, faithfulness ensures that this reduction is optimal: once we have information on $\mathcal{V} \cap \partial \widetilde{\Gamma}$, no further reduction is possible.

\bibliographystyle{plainnat}

\bibliography{../../Bib/unified_graph_bib}

\begin{thebibliography}{32}
\providecommand{\natexlab}[1]{#1}
\providecommand{\url}[1]{\texttt{#1}}
\expandafter\ifx\csname urlstyle\endcsname\relax
  \providecommand{\doi}[1]{doi: #1}\else
  \providecommand{\doi}{doi: \begingroup \urlstyle{rm}\Url}\fi

\bibitem[Anderes et~al.(2020)Anderes, M{\o}ller, and Rasmussen]{anderes2020isotropic}
Ethan Anderes, Jesper M{\o}ller, and Jakob~G Rasmussen.
\newblock Isotropic covariance functions on graphs and their edges.
\newblock \emph{Ann. Statist.}, 48\penalty0 (4):\penalty0 2478--2503, 2020.

\bibitem[Ash and Gardner(1975)]{ashtopics}
Robert~B. Ash and Melvin~F. Gardner.
\newblock \emph{Topics in stochastic processes}.
\newblock Probability and Mathematical Statistics, Vol. 27. Academic Press [Harcourt Brace Jovanovich, Publishers], New York-London, 1975.

\bibitem[Bolin et~al.(2023)Bolin, Simas, and Wallin]{BSW2023_AOS}
David Bolin, Alexandre~B. Simas, and Jonas Wallin.
\newblock Statistical inference for {G}aussian {W}hittle--{M}at\'{e}rn fields on metric graphs.
\newblock Preprint, arXiv:2304.10372, 2023.

\bibitem[Bolin et~al.(2024{\natexlab{a}})Bolin, Kov\'{a}cs, Kumar, and Simas]{bolinetal_fem_graph}
David Bolin, M.~Kov\'{a}cs, V.~Kumar, and Alexandre~B. Simas.
\newblock Regularity and numerical approximation of fractional elliptic differential equations on compact metric graphs.
\newblock \emph{Math. Comp.}, 93:\penalty0 2439--2472, 2024{\natexlab{a}}.

\bibitem[Bolin et~al.(2024{\natexlab{b}})Bolin, Simas, and Wallin]{BSW2022}
David Bolin, Alexandre~B. Simas, and Jonas Wallin.
\newblock Gaussian {W}hittle-{M}at\'ern fields on metric graphs.
\newblock \emph{Bernoulli}, 30:\penalty0 1611--1639, 2024{\natexlab{b}}.

\bibitem[Bolin et~al.(2024{\natexlab{c}})Bolin, Simas, and Wallin]{BSW_Markov}
David Bolin, Alexandre~B. Simas, and Jonas Wallin.
\newblock Markov properties of {G}aussian random fields on compact metric graphs.
\newblock \emph{Bernoulli}, 2024{\natexlab{c}}.
\newblock \doi{arxiv:2304.03190}.
\newblock Accepted for publication.

\bibitem[Borisov(1983)]{borisov_markov}
I.~S. Borisov.
\newblock On a criterion for gaussian random processes to be markovian.
\newblock \emph{Theory of Probability \& Its Applications}, 27\penalty0 (4):\penalty0 863--865, 1983.

\bibitem[Borovitskiy et~al.(2021)Borovitskiy, Azangulov, Terenin, Mostowsky, Deisenroth, and Durrande]{borovitskiy2021matern}
Viacheslav Borovitskiy, Iskander Azangulov, Alexander Terenin, Peter Mostowsky, Marc Deisenroth, and Nicolas Durrande.
\newblock Mat{\'e}rn {G}aussian processes on graphs.
\newblock In \emph{International Conference on Artificial Intelligence and Statistics}, pages 2593--2601. PMLR, 2021.

\bibitem[Bullmore and Sporns(2012)]{Bullmore2012}
Ed~Bullmore and Olaf Sporns.
\newblock The economy of brain network organization.
\newblock \emph{Nat. Rev. Neurosci.}, 13\penalty0 (5):\penalty0 336--349, 2012.
\newblock \doi{10.1038/nrn3214}.

\bibitem[Chaudhuri et~al.(2023)Chaudhuri, Saez, Varga, and Juan]{CHAUDHURI2023}
Somnath Chaudhuri, Marc Saez, Diego Varga, and Pablo Juan.
\newblock Spatiotemporal modeling of traffic risk mapping: A study of urban road networks in barcelona, spain.
\newblock \emph{Spat. Stat.}, 53:\penalty0 100722, 2023.

\bibitem[Cressie et~al.(2006)Cressie, Frey, Harch, and Smith]{CressieRiver}
Noel Cressie, Jesse Frey, Bronwyn Harch, and Mick Smith.
\newblock Spatial prediction on a river network.
\newblock \emph{J. Agric. Biol. Environ. Stat.}, 11\penalty0 (2):\penalty0 127, 2006.

\bibitem[Fallat et~al.(2017)Fallat, Lauritzen, Sadeghi, Uhler, Wermuth, and Zwiernik]{fallat2017total}
Shaun Fallat, Steffen Lauritzen, Kayvan Sadeghi, Caroline Uhler, Nanny Wermuth, and Piotr Zwiernik.
\newblock Total positivity in markov structures.
\newblock \emph{The Annals of Statistics}, pages 1152--1184, 2017.

\bibitem[Gardner et~al.(2003)Gardner, Sullivan, and Lembo]{Gardner2003}
Beth Gardner, Patrick~J Sullivan, and Arthur~J Lembo, Jr.
\newblock Predicting stream temperatures: geostatistical model comparison using alternative distance metrics.
\newblock \emph{Can. J. Fish. Aquat. Sci.}, 60\penalty0 (3):\penalty0 344--351, 2003.

\bibitem[He et~al.(2018)He, Stankovic, Liao, and Stankovic]{He2018}
Kanghang He, Lina Stankovic, Jing Liao, and Vladimir Stankovic.
\newblock Non-intrusive load disaggregation using graph signal processing.
\newblock \emph{IEEE Trans. Smart Grid}, 9\penalty0 (3):\penalty0 1739--1747, 2018.

\bibitem[Inc.()]{Mathematica}
Wolfram~Research{,} Inc.
\newblock Mathematica, {V}ersion 14.1.
\newblock Champaign, IL, 2024.

\bibitem[Isaak et~al.(2014)Isaak, Peterson, Ver~Hoef, Wenger, Falke, Torgersen, Sowder, Steel, Fortin, Jordan, Ruesch, Som, and Monestiez]{Isaak2014}
Daniel~J. Isaak, Erin~E. Peterson, Jay~M. Ver~Hoef, Seth~J. Wenger, Jeffrey~A. Falke, Christian~E. Torgersen, Colin Sowder, E.~Ashley Steel, Marie-Josee Fortin, Chris~E. Jordan, Aaron~S. Ruesch, Nicholas Som, and Pascal Monestiez.
\newblock Applications of spatial statistical network models to stream data.
\newblock \emph{WIREs Water}, 1\penalty0 (3):\penalty0 277--294, 2014.

\bibitem[Janson(1997)]{janson_gaussian}
Svante Janson.
\newblock \emph{Gaussian {H}ilbert spaces}, volume 129 of \emph{Cambridge Tracts in Mathematics}.
\newblock Cambridge University Press, Cambridge, 1997.
\newblock ISBN 0-521-56128-0.

\bibitem[K\"{u}nsch(1979)]{kunsch}
H.~K\"{u}nsch.
\newblock Gaussian {M}arkov random fields.
\newblock \emph{J. Fac. Sci. Univ. Tokyo Sect. IA Math.}, 26\penalty0 (1):\penalty0 53--73, 1979.
\newblock ISSN 0040-8980.

\bibitem[Lauritzen(1996)]{lauritzen1996graphical}
Steffen~L Lauritzen.
\newblock \emph{Graphical models}, volume~17.
\newblock Clarendon Press, 1996.

\bibitem[Lindgren et~al.(2011)Lindgren, Rue, and Lindstr\"{o}m]{lindgren11}
Finn Lindgren, H{\aa}vard Rue, and Johan Lindstr\"{o}m.
\newblock An explicit link between {G}aussian fields and {G}aussian {M}arkov random fields: the stochastic partial differential equation approach.
\newblock \emph{J.\ R.\ Stat.\ Soc.\ Ser.\ B Stat.\ Methodol.}, 73\penalty0 (4):\penalty0 423--498, 2011.
\newblock With discussion and a reply by the authors.

\bibitem[Mandrekar(1976)]{Mandrekar1976}
V.~Mandrekar.
\newblock Germ-field {M}arkov property for multiparameter processes.
\newblock In \emph{S\'{e}minaire de {P}robabilit\'{e}s, {X}}, Lecture Notes in Math., Vol. 511, pages 78--85. Springer, Berlin, 1976.

\bibitem[McGill et~al.(2021)McGill, Brooks, and Steel]{mcgill2021spatiotemporal}
Lillian~M McGill, J~Ren{\'e}e Brooks, and E~Ashley Steel.
\newblock Spatiotemporal dynamics of water sources in a mountain river basin inferred through $\delta^2$h and $\delta^{18}$o of water.
\newblock \emph{Hydrological processes}, 35\penalty0 (3):\penalty0 e14063, 2021.

\bibitem[M{\o}ller and Rasmussen(2024)]{moller2022lgcp}
Jesper M{\o}ller and Jakob~G. Rasmussen.
\newblock Cox processes driven by transformed {G}aussian processes on linear networks-a review and new contributions.
\newblock \emph{Scand. J. Stat.}, 51\penalty0 (3):\penalty0 1288--1322, 2024.

\bibitem[Moradi and Sharifi(2024)]{MORADI2024}
Mehdi Moradi and Ali Sharifi.
\newblock Summary statistics for spatio-temporal point processes on linear networks.
\newblock \emph{Spat. Stat.}, 61:\penalty0 100840, 2024.

\bibitem[Ortega et~al.(2018)Ortega, Frossard, Kova{\v{c}}evi{\'c}, Moura, and Vandergheynst]{ortega2018graph}
Antonio Ortega, Pascal Frossard, Jelena Kova{\v{c}}evi{\'c}, Jos{\'e}~MF Moura, and Pierre Vandergheynst.
\newblock Graph signal processing: Overview, challenges, and applications.
\newblock \emph{Proc. IEEE}, 106\penalty0 (5):\penalty0 808--828, 2018.

\bibitem[Perozzi et~al.(2014)Perozzi, Al-Rfou, and Skiena]{Perozzi2014}
Bryan Perozzi, Rami Al-Rfou, and Steven Skiena.
\newblock Deepwalk: online learning of social representations.
\newblock In \emph{Proceedings of the 20th ACM SIGKDD International Conference on Knowledge Discovery and Data Mining}, KDD '14, page 701–710. Association for Computing Machinery, 2014.
\newblock ISBN 9781450329569.

\bibitem[Rue and Held(2005)]{rue2005gaussian}
H{\aa}vard Rue and Leonhard Held.
\newblock \emph{Gaussian {M}arkov random fields}, volume 104 of \emph{Monographs on Statistics and Applied Probability}.
\newblock Chapman \& Hall/CRC, Boca Raton, FL, 2005.
\newblock ISBN 978-1-58488-432-3; 1-58488-432-0.
\newblock Theory and applications.

\bibitem[Slawski and Hein(2015)]{slawski2015estimation}
Martin Slawski and Matthias Hein.
\newblock Estimation of positive definite m-matrices and structure learning for attractive gaussian markov random fields.
\newblock \emph{Linear Algebra and its Applications}, 473:\penalty0 145--179, 2015.

\bibitem[Valdivia et~al.(2015)Valdivia, Dias, Petronetto, Silva, and Nonato]{Valdivia2015}
Paola Valdivia, Fabio Dias, Fabiano Petronetto, Cláudio~T. Silva, and L.~G. Nonato.
\newblock Wavelet-based visualization of time-varying data on graphs.
\newblock In \emph{2015 IEEE Conference on Visual Analytics Science and Technology (VAST)}, pages 1--8, 2015.

\bibitem[Ver~Hoef and Peterson(2010)]{Hoef2010}
Jay~M. Ver~Hoef and Erin~E. Peterson.
\newblock A moving average approach for spatial statistical models of stream networks.
\newblock \emph{J. Amer. Statist. Assoc.}, 105\penalty0 (489):\penalty0 6--18, 2010.

\bibitem[Ver~Hoef et~al.(2006)Ver~Hoef, Peterson, and Theobald]{Hoef2006}
Jay~M. Ver~Hoef, Erin Peterson, and David Theobald.
\newblock Spatial statistical models that use flow and stream distance.
\newblock \emph{Environ. Ecol. Stat.}, 13\penalty0 (4):\penalty0 449--464, 2006.

\bibitem[Wagner et~al.(2005)Wagner, Choi, Baraniuk, and Delouille]{Wagner2005}
R.~Wagner, Hyeokho Choi, R.~Baraniuk, and V.~Delouille.
\newblock Distributed wavelet transform for irregular sensor network grids.
\newblock In \emph{IEEE/SP 13th Workshop on Statistical Signal Processing, 2005}, pages 1196--1201, 2005.

\end{thebibliography}
\end{document}